\documentclass[a4paper, 12pt, oneside]{amsart}

\usepackage{amsmath}
\usepackage{amsthm}
\newtheorem{theorem}{Theorem}
\newtheorem{corollary}{Corollary}

\begin{document}
	\pagestyle{plain}
	\title{General Solution of Second Order Linear Ordinary Differential Equation}
	\author{Rajnish Kumar Jha}
	\email{rajnishkrjha.2020@gmail.com}
	
	\begin{abstract}
		In this paper we present a direct formula for the solution of the general second order linear ordinary differential equation as our main result such that the parameters required for the formula are determined using another differential equation, which happens to be a Riccati Equation. We also derive other results based on the main result which include special cases for the concerned differential equation with variable coefficients, formula for solution of concerned differential equation with constant coefficients and formula for the solution of the concerned differential equation with one complementary solution known.
	\end{abstract} 

	\maketitle

	\section[intro]{Introduction}
		
	The general second order linear ordinary differential equation \cite{AEM2} is given as 
		\begin{equation}
		y_{(2)} + B(x)y_{(1)} + C(x)y = R(x),
		\label{1}
		\end{equation}	
	where $ y_{(k)} $ represents $ k^{th} $ derivative of the dependent variable - $ y $ with respect to the independent variable $ x $ i.e. 
		\begin{equation*}
		 y_{(k)} = \frac{d^{k}y}{dx^k} ,
		\end{equation*}		
	$B(x), C(x)$ and $R(x)$ are functions of the independent variable $x$.
	
	When $R(x)=0$, Eqn.~\ref{1} represents a homogeneous second order linear ordinary differential equation and when $R(x) \neq 0$, it represents a non-homogeneous second order linear ordinary differential equation.

	When the coefficients $B(x)$ and $C(x)$ are constants then Eqn.~\ref{1} is known as second order linear ordinary differential equation with constant coefficients, otherwise it is known as second order linear ordinary differential equation with variable coefficients. (See \cite{AEM2})
	
	In this paper, we present a method to solve the general second order linear ordinary differential equation wherein we derive an expression for the solution having integrating factors involved and also determine another differential equation in this regard (which is a Riccati Equation \cite{AEM1}) such that when this equation can be solved, the solution of the corresponding second order differential equation can also be obtained.
	
	When the coefficients $B(x)$ and $C(x)$ are not constants then the differential equation given by Eqn.~\eqref{1} can be difficult to solve. However, under certain cases where the coefficients satisfy some conditions, we can determine the exact solution of such differential equation. Two such cases are described in this paper. (See Corollary-\ref{Col1} and Corollary-\ref{{Col2}} )
	
	We also derive a formula for the complete solution of the differential equation given by Eqn.~\eqref{1} provided we know one complementary solution of the given equation (See Corollary-\ref{Col4}). Moreover, when the coefficients of the differential equation given by Eqn.~\eqref{1} are constant then the complete solution can be determined exactly as it is discussed in this paper. (See Corollary-\ref{Col3} )
	
	The main result is stated as Theorem~\ref{Thm} while the other results are stated in form of its corollaries.

	\section[main-result]{Main Result}
	
	In this section we give the main result in form of a Theorem , which gives the solution expression for the differential equation given by Eqn.~\eqref{1} in terms of $R(x)$ and Integrating Factors. We present Theorem~\ref{Thm} along with the proof here as follows
	
	\begin{theorem}
		\label{Thm}
		The solution of second order linear ordinary differential equation, as described by Eqn. \eqref{1}, is given as		
		\begin{equation}		
		y = \frac{1}{h(x)}\!\!\int\!\!\frac{h(x)}{g(x)}\!\!\int\!\! g(x)R(x)\, dx dx + \frac{C_1}{h(x)}\!\!\int\!\!\frac{h(x)}{g(x)}\,dx + \frac{C_2}{h(x)} \label{2}		
		\end{equation} 
		where
		\begin{gather}
		g(x) = e^{\int P(x) \, dx},\label{3} \\
		h(x) = e^{\int Q(x) \, dx}, \label{4} \\
		P(x) + Q(x) = B(x), \label{5} \\
		Q'(x) = {Q(x)}^2 -B(x)Q(x) + C(x) \label{6},
		\end{gather}
		$Q'(x)$ represents first order derivative of $Q(x)$ with respect to $x$ and $g(x)$ and $h(x)$ are the first and second Integrating Factors respectively.
	\end{theorem}

	\begin{proof}
		Consider the differential equation given by Eqn.~\eqref{1},
		\begin{equation*}
		y_{(2)} + B(x)y_{(1)} + C(x)y = R(x).
		\end{equation*}		
		Let		
		\begin{equation}
		 B(x) = P(x) + Q(x),
		 \label{7}
		\end{equation}
		where $P(x)$ and $Q(x)$ are the functions of the independent variable $x$. Then Eqn.~\eqref{1} becomes
		\begin{equation}
		y_{(2)} + P(x)y_{(1)} + Q(x)y_{(1)} + C(x)y = R(x).
		\label{8}
		\end{equation}
		We multiply Eqn.~\eqref{8} by a function $g(x)$, which we call here as the first Integrating Factor, to get 
		\begin{equation}
		(g(x)y_{(2)}+g(x)P(x)y_{(1)}) + (g(x)Q(x)y_{(1)} + g(x)C(x)y) = g(x)R(x).
		\label{9}
		\end{equation}
		We impose the following conditions on $g(x)$,
		\begin{gather}
		g'(x) = g(x)P(x), \label{10} \\
		(g(x)Q(x))' = g(x)C(x), \label{11}
		\end{gather}
		such that Eqn.~\eqref{9} becomes
		\begin{equation}
		\frac{d}{dx}(g(x)y_{(1)}) + \frac{d}{dx}(g(x)Q(x)y) = g(x)R(x).
		\label{12}
		\end{equation}
		On integrating Eqn.~\eqref{12} we get
		\begin{gather}
		g(x)y_{(1)} + g(x)y = \!\! \int \!\! g(x)R(x) \, dx + C_1 \label{13} \\
		\implies y_{(1)} + Q(x)y = \frac{1}{g(x)} \!\! \int \!\! g(x)R(x) \, dx + \frac{C_1}{g(x)} \label{14}
		\end{gather}
		Now, we multiply Eqn.~\eqref{14} with another function $h(x)$, which we call here as the second Integrating Factor, to get
		\begin{equation}
		h(x)y_{(1)} + h(x)Q(x)y = \frac{h(x)}{g(x)}\!\!\int \!\! g(x)R(x)\,dx + C_1\frac{h(x)}{g(x)}.
		\label{15}
		\end{equation}
		We impose the following condition on $h(x)$,
		\begin{equation}
		h'(x) = h(x)Q(x),
		\label{16}
		\end{equation}
		such that Eqn.~\eqref{15} transforms into 
		\begin{equation}
		\frac{d}{dx}(h(x)y) = \frac{h(x)}{g(x)} \!\! \int \!\! g(x)R(x) \, dx + C_1\frac{h(x)}{g(x)}.
		\label{17}
		\end{equation}
		On integrating Eqn.~\eqref{17}, we get
		\begin{gather}
		h(x)y =\!\!\int\!\!\frac{h(x)}{g(x)}\!\!\int\!\!g(x)R(x) \,dxdx + C_1\!\!\int \!\! \frac{h(x)}{g(x)} \, dx + C_2 \label{18} \\
		\implies y = \frac{1}{h(x)}\!\!\int\!\!\frac{h(x)}{g(x)}\!\!\int\!\!g(x)R(x) \,dxdx + \frac{C_1}{h(x)}\!\!\int \!\! \frac{h(x)}{g(x)} \, dx + \frac{C_2}{h(x)}, \label{19}
		\end{gather}
		such that Eqn.~\eqref{19} gives the solution for the differential equation given by Eqn.~\eqref{1} where $g(x)$ and $h(x)$ are determined using conditions given in Eqn.~\eqref{10}, Eqn.~\eqref{11} and Eqn.~\eqref{16}.
		
		The conditions described in Eqn.~\eqref{10}, Eqn.~\eqref{11} and Eqn.~\eqref{16} can also be expressed as
		\begin{gather}
		g'(x) = g(x)P(x) \implies g(x) = e^{\int P(x)\,dx} \label{20} \\
		h'(x) = h(x)Q(x) \implies h(x) = e^{\int Q(x)\,dx} \label{21} \\
		\begin{split}
		(g(x)Q(x))' &= g(x)C(x) \\ \implies g'(x)Q(x) + g(x)Q'(x) &= g(x)C(x). \label{22}
		\end{split}
		\end{gather}
		On using Eqn.~\eqref{10} in Eqn.~\eqref{22}, we get
		\begin{gather}
		g(x)P(x)Q(x) + g(x)Q'(x) = g(x)C(x) \label{23} \\
		\implies Q'(x) + P(x)Q(x) = C(x). \label{24}
		\end{gather}
		Also, by using Eqn.~\eqref{7} in Eqn.~\eqref{24}, we can write an equivalent equation as 
		\begin{gather}
		Q'(x) + B(x)Q(x) - {Q(x)}^2 = C(x) \label{25} \\
		\implies Q'(x) = {Q(x)}^2 - B(x)Q(x) + C(x) \label{26}
		\end{gather}
		which is a Riccati Equation.
		
		Thus, the equations \eqref{7}, \eqref{20}, \eqref{21} and \eqref{26} can be used to determine the Integrating Factors $g(x)$ and $h(x)$ such that they also determine the solution expression given by Eqn.~\eqref{19} for the differential equation given by Eqn.~\eqref{1}.
		
		Therefore, the solution of the general second order linear ordinary differential equation \eqref{1} is given by \eqref{19}, i.e.,
		\begin{equation*}
		y = \frac{1}{h(x)}\!\!\int\!\!\frac{h(x)}{g(x)}\!\!\int\!\!g(x)R(x) \,dxdx + \frac{C_1}{h(x)}\!\!\int \!\! \frac{h(x)}{g(x)} \, dx + \frac{C_2}{h(x)},
		\end{equation*}
		where
		\begin{gather*}
		g(x) = e^{\int P(x) \, dx}, \\
		h(x) = e^{\int Q(x) \, dx}, \\
		P(x) + Q(x) = B(x), \\
		Q'(x) = {Q(x)}^2 -B(x)Q(x) + C(x) \, \text{and}
		\end{gather*}
		$g(x)$ and $h(x)$ are the first and second Integrating Factors respectively.
		
	\end{proof}

	It is worthwhile to discuss about the complementary solution part and the particular integral part of the solution of the differential equation \eqref{1} given by Eqn.~\eqref{19}.

	The complementary solution part of a general linear ordinary differential equation is given by the solution of the corresponding homogeneous equation. So, in the context of the differential equation \eqref{1}, we get the corresponding homogeneous equation by putting $R(x)=0$ , i.e.
	\begin{equation}
	y_{(2)} + B(x)y_{(1)} + C(x)y = 0 \label{27}
	\end{equation}
	such that by Theorem~\ref{Thm}, its solution is given as 
	\begin{gather}
	y = \frac{1}{h(x)}\!\!\int\!\!\frac{h(x)}{g(x)}\!\!\int\!\!0 \,dxdx + \frac{C_1}{h(x)}\!\!\int \!\! \frac{h(x)}{g(x)} \, dx + \frac{C_2}{h(x)} \label{28}\\
	\begin{split}
	\implies y &= \frac{(C_0+C_1)}{h(x)}\!\!\int \!\! \frac{h(x)}{g(x)} \, dx + \frac{C_2}{h(x)} \\ &= \frac{C}{h(x)}\!\!\int \!\! \frac{h(x)}{g(x)} \, dx + \frac{C_2}{h(x)} \,(C=C_0+C_1).\label{29}
	\end{split}
	\end{gather}
	From the above observations, we conclude that the complementary solution part of the solution of the differential equation \eqref{1} given by Eqn.~\eqref{19} is the part containing arbitary constants while the term having no arbitary constant gives the particular solution (or the non-homogeneous solution) part of the solution.

	\section[other-results]{Other Results Based on the Main Result}

	In this section we state and prove some other important results related to second order linear ordinary differential equations and based on Theorem~\ref{Thm} in form of corollaries.

	The first two results deal with special cases of the general second order linear ordinary differential equation having variable coefficients where the coefficients satisfy some condition. The two results are described here as follows
	
	\begin{corollary}
		\label{Col1}
		If the coefficients of the second order linear ordinary differential equation \eqref{1} satisfy the condition
		\begin{equation}
		{B(x)}^2 + 2B'(x) -4C(x) = 0,
		\label{30}
		\end{equation}
		then its solution is given as
		\begin{equation} 
		\begin{split}
		y=(x+c)e^{-\!\!\int\!\!\frac{B(x)}{2}\,dx}\!\!\int\!\!\frac{1}{(x+c)^2}\!\!\int\!\!(x+c)e^{\!\!\int\!\!\frac{B(x)}{2}\,dx}R(x)\,dxdx +\\+ C_1 e^{-\!\!\int\!\!\frac{B(x)}{2}} + C_2xe^{-\!\!\int\!\!\frac{B(x)}{2}\,dx},
		\label{31}
		\end{split}
		\end{equation}
		where c is any constant.
	\end{corollary}

	\begin{proof}
		Consider the general second order linear ordinary differential equation given by Eqn.~\eqref{1},
		\begin{equation}
		y_{(2)} + B(x)y_{(1)} + C(x)y = R(x).
		\label{32}
		\end{equation}
		By Theorem~\ref{Thm}, this equation has solution given as
		\begin{equation}
		y = \frac{1}{h(x)}\!\!\int\!\!\frac{h(x)}{g(x)}\!\!\int\!\!g(x)R(x) \,dxdx + \frac{C_1}{h(x)}\!\!\int \!\! \frac{h(x)}{g(x)} \, dx + \frac{C_2}{h(x)},
		\label{33}
		\end{equation}
		where
		\begin{gather}
		g(x) = e^{\int P(x) \, dx}, \label{34} \\ 
		h(x) = e^{\int Q(x) \, dx}, \label{35} \\
		P(x) + Q(x) = B(x), \label{36} \\
		Q'(x) = {Q(x)}^2 -B(x)Q(x) + C(x) \, \label{37} \text{and}
		\end{gather}
		$g(x)$ and $h(x)$ are the first and second Integrating Factors respectively.\\ \\
		Let
		\begin{gather}
		P(x) = \frac{B(x)}{2} - f(x) \label{38} \\
		\implies Q(x) = \frac{B(x)}{2} + f(x). \label{39}
		\end{gather}
		Then Eqn.~\eqref{37} becomes
		\begin{gather}
		\begin{split}
		\frac{B'(x)}{2} + f'(x) &= \frac{{B(x)}^2}{4} + {f(x)}^2\\ &+ B(x)f(x) - \frac{{B(x)}^2}{2} - B(x)f(x) + C(x) \label{40}
		\end{split}\\
		\implies f'(x) = {f(x)}^2 - \frac{{B(x)}^2 + 2B'(x) -4C(x)}{4}. \label{41}
		\end{gather}
		If the coefficients $B(x)$ and $C(x)$ satisfy the equation 
		\begin{equation}
		{B(x)}^2 + 2B'(x) -4C(x) = 0,
		\label{42}
		\end{equation}
		then Eqn.~\eqref{41} becomes
		\begin{equation}
		f'(x) = {f(x)}^2,
		\label{43}
		\end{equation}
		which has solution expressed as 
		\begin{equation}
		f(x) = -\frac{1}{(x+c)} \, \text{($c$ is an arbitary constant)}
		\label{44}
		\end{equation}
		On using $f(x)$ in Eqn.~\eqref{38} and Eqn.~\eqref{39}, we get
		\begin{gather}
		P(x) = \frac{B(x)}{2} + \frac{1}{(x+c)} \label{45} \\
		Q(x) = \frac{B(x)}{2} - \frac{1}{(x+c)} \label{46}
		\end{gather}
		Then on using values of $P(x)$ and $Q(x)$ in Eqn.~\eqref{34} and Eqn.~\eqref{35} respectively, we get the Integrating factors as
		\begin{gather}
		g(x) = (x+c)e^{\int\frac{B(x)}{2}\, dx} \label{47} \\
		h(x) = \frac{1}{(x+c)}e^{\int\frac{B(x)}{2}\, dx}, \label{48}
		\end{gather}
		such that the solution expression becomes
		\begin{gather}
		\begin{split}
		y=(x+c)e^{-\!\!\int\!\!\frac{B(x)}{2}\,dx}\!\!\int\!\!\frac{1}{(x+c)^2}\!\!\int\!\!(x+c)e^{\!\!\int\!\!\frac{B(x)}{2}\,dx}R(x)\,dxdx +\\+ C_1(x+c) e^{-\!\!\int\!\!\frac{B(x)}{2}}\!\!\int\!\!\frac{1}{{(x+c)}^2}\,dx + C_2(x+c)e^{-\!\!\int\!\!\frac{B(x)}{2}\,dx},
		\label{49}
		\end{split}\\
		\begin{split}
		\implies y=(x+c)e^{-\!\!\int\!\!\frac{B(x)}{2}\,dx}\!\!\int\!\!\frac{1}{(x+c)^2}\!\!\int\!\!(x+c)e^{\!\!\int\!\!\frac{B(x)}{2}\,dx}R(x)\,dxdx +\\+ C_0 e^{-\!\!\int\!\!\frac{B(x)}{2}} + C_2xe^{-\!\!\int\!\!\frac{B(x)}{2}\,dx}\,(C_0 = C_2 c-C_1),
		\label{50}
		\end{split}
		\end{gather}
		where Eqn.~\eqref{50} gives the required solution expression for the differential equation \eqref{32} when the condition given by Eqn.~\eqref{42} is true.
	\end{proof}
	
	\begin{corollary}
		\label{{Col2}}
		If the coefficients of the second order linear ordinary differential equation \eqref{1} satisfy the condition
		\begin{equation}
		{B(x)}^2 + 2B'(x) -4C(x) = k\, \text{($k$ is a constant, $k \neq 0$)},
		\label{51}
		\end{equation}
		then its solution is given as
		\begin{small}
		\begin{equation} 
		\begin{split}
		&y \! = \! e^{-\!\! \int \!\! \frac{B(x)}{2}\,dx}(e^{\sqrt{k}(x+c)}\!\!\!-\!e^{-\sqrt{k}(x+c)}) \\
		&\times \int \!\! \frac{1}{{(e^{\sqrt{k}(x+c)}\!\!\!-\!e^{-\sqrt{k}(x+c)})}^2} \!\! \int \!\! e^{\int \frac{B(x)}{2}\,dx} (e^{\sqrt{k}(x+c)}\!\!\!-\!e^{-\sqrt{k}(x+c)}) R(x) \,dxdx \\
		&+C_1e^{-(\int\frac{B(x)}{2}\,dx + \sqrt{k}(x+c))} + C_2e^{(-\int\frac{B(x)}{2}\,dx + \sqrt{k}(x+c))}
		\label{52}
		\end{split}
		\end{equation}
		\end{small}
		where c is any constant.
	\end{corollary}

	\begin{proof}
		Consider the general second order linear ordinary differential equation given by Eqn.~\eqref{1},
		\begin{equation}
		y_{(2)} + B(x)y_{(1)} + C(x)y = R(x).
		\label{53}
		\end{equation}
		By Theorem~\ref{Thm}, this equation has solution given as
		\begin{equation}
		y = \frac{1}{h(x)}\!\!\int\!\!\frac{h(x)}{g(x)}\!\!\int\!\!g(x)R(x) \,dxdx + \frac{C_1}{h(x)}\!\!\int \!\! \frac{h(x)}{g(x)} \, dx + \frac{C_2}{h(x)},
		\label{54}
		\end{equation}
		where
		\begin{gather}
		g(x) = e^{\int P(x) \, dx}, \label{55} \\ 
		h(x) = e^{\int Q(x) \, dx}, \label{56} \\
		P(x) + Q(x) = B(x), \label{57} \\
		Q'(x) = {Q(x)}^2 -B(x)Q(x) + C(x) \, \label{58} \text{and}
		\end{gather}
		$g(x)$ and $h(x)$ are the first and second Integrating Factors respectively.\\ \\
		Let
		\begin{gather}
		P(x) = \frac{B(x)}{2} - f(x) \label{59} \\
		\implies Q(x) = \frac{B(x)}{2} + f(x). \label{60}
		\end{gather}
		Then Eqn.~\eqref{58} becomes
		\begin{gather}
		\begin{split}
		\frac{B'(x)}{2} + f'(x) &= \frac{{B(x)}^2}{4} + {f(x)}^2\\ &+ B(x)f(x) - \frac{{B(x)}^2}{2} - B(x)f(x) + C(x) \label{61}
		\end{split}\\
		\implies f'(x) = {f(x)}^2 - \frac{{B(x)}^2 + 2B'(x) -4C(x)}{4}. \label{62}
		\end{gather}
		If the coefficients $B(x)$ and $C(x)$ satisfy the equation
		\begin{equation}
		\frac{{B(x)}^2 + 2B'(x) -4C(x)}{4} = k\, \text{($k$ is a constant, $k \neq 0$)},
		\label{63}
		\end{equation}
		then Eqn.~\eqref{62} becomes
		\begin{equation}
		f'(x) = {f(x)}^2 - k .
		\label{64}
		\end{equation}
		The solution of the above equation is given as
		\begin{equation}
		f(x) = -\sqrt{k}\left(\frac{e^{\sqrt{k}(x+c)} + e^{-\sqrt{k}(x+c)}}{e^{\sqrt{k}(x+c)} - e^{-\sqrt{k}(x+c)}}\right),
		\label{65}
		\end{equation}
		where $c$ is an arbitary constant.
		
		On using $f(x)$ in Eqn.~\eqref{59} and Eqn.~\eqref{60}, we obtain
		\begin{gather}
		P(x) = \frac{B(x)}{2} + \sqrt{k}\left(\frac{e^{\sqrt{k}(x+c)} + e^{-\sqrt{k}(x+c)}}{e^{\sqrt{k}(x+c)} - e^{-\sqrt{k}(x+c)}}\right), \label{66} \\
		Q(x) = \frac{B(x)}{2} - \sqrt{k}\left(\frac{e^{\sqrt{k}(x+c)} + e^{-\sqrt{k}(x+c)}}{e^{\sqrt{k}(x+c)} - e^{-\sqrt{k}(x+c)}}\right). \label{67}
		\end{gather}
		Then on using $P(x)$ and $Q(x)$ in equations \eqref{55} and \eqref{56} respectively, we obtain the Integrating Factors $g(x)$ and $h(x)$ as 
		\begin{gather}
		\begin{split}
		g(x) = e^{\int P(x) \, dx} &= e^{\int \frac{B(x)}{2} \, dx}e^{\int \sqrt{k}\left(\frac{e^{\sqrt{k}(x+c)} + e^{-\sqrt{k}(x+c)}}{e^{\sqrt{k}(x+c)} - e^{-\sqrt{k}(x+c)}}\right)\,dx} \\
		&= e^{\int \frac{B(x)}{2} \, dx}\left(e^{\sqrt{k}(x+c)} - e^{-\sqrt{k}(x+c)}\right),
		\label{68}
		\end{split} \\
		\begin{split}
		h(x) = e^{\int Q(x) \, dx} &= e^{\int \frac{B(x)}{2} \, dx}e^{-\int \sqrt{k}\left(\frac{e^{\sqrt{k}(x+c)} + e^{-\sqrt{k}(x+c)}}{e^{\sqrt{k}(x+c)} - e^{-\sqrt{k}(x+c)}}\right)\,dx} \\
		&= e^{\int \frac{B(x)}{2} \, dx}\frac{1}{\left(e^{\sqrt{k}(x+c)} - e^{-\sqrt{k}(x+c)}\right)}.
		\label{69}
		\end{split}
		\end{gather}
		On using $g(x)$ and $h(x)$ in Eqn.~\eqref{54}, the solution expression becomes
		\begin{small}
		\begin{gather} 
		\begin{split}
		&y \!= \! e^{-\!\! \int \!\! \frac{B(x)}{2}\,dx}(e^{\sqrt{k}(x+c)}\!\!\!-\!e^{-\sqrt{k}(x+c)})\\
		&\times\int \!\! \frac{1}{{(e^{\sqrt{k}(x+c)}\!\!\!-\!e^{-\sqrt{k}(x+c)})}^2} \!\! \int \!\! e^{\int \frac{B(x)}{2}\,dx} (e^{\sqrt{k}(x+c)}\!\!\!-\!e^{-\sqrt{k}(x+c)}) R(x) \,dxdx \\
		&+C_1 e^{-\int \frac{B(x)}{2} \, dx}\left(e^{\sqrt{k}(x+c)} - e^{-\sqrt{k}(x+c)}\right) \!\! \int \!\! \frac{1}{{\left(e^{\sqrt{k}(x+c)} - e^{-\sqrt{k}(x+c)}\right)}^2} \, dx \\
		&+ C_2e^{-\int \frac{B(x)}{2} \, dx}\left(e^{\sqrt{k}(x+c)} - e^{-\sqrt{k}(x+c)}\right) 
		\label{70}
		\end{split} \\
		\nonumber \\ 
		\nonumber \\
		\begin{split}
		&\implies \!y \!=\! e^{-\!\! \int \!\! \frac{B(x)}{2}\,dx}(e^{\sqrt{k}(x+c)}\!\!\!-\!e^{-\sqrt{k}(x+c)})\\
		&\times\int \!\! \frac{1}{{(e^{\sqrt{k}(x+c)}\!\!\!-\!e^{-\sqrt{k}(x+c)})}^2} \!\! \int \!\! e^{\int \frac{B(x)}{2}\,dx} (e^{\sqrt{k}(x+c)}\!\!\!-\!e^{-\sqrt{k}(x+c)}) R(x) \,dxdx \\
		&+C_0e^{-(\int\frac{B(x)}{2}\,dx + \sqrt{k}(x+c))}\\
		&+ C_2e^{(-\int\frac{B(x)}{2}\,dx + \sqrt{k}(x+c))} \\
		&\,\,(C_0 = -\frac{C_1}{2\sqrt{k}}-C_2)
		\label{71}
		\end{split}
		\end{gather}
		\end{small}
		such that Eqn.~\eqref{71} gives the required solution expression for the differential equation \eqref{53} when the condition given by Eqn.~\eqref{63} is true.
	\end{proof}

	The third result deals with second order linear ordinary differential equation having constant coefficients. In this case exact solution is completely determined. The result is stated in form of a corollary of Theorem \ref{Thm} presented here along with the proof as follows
	
	\begin{corollary}
		If the coefficients of the second order linear ordinary differential equation \eqref{1} are constant such that $B(x)=B$ and $C(x)=C$ then its solution is given by
		\begin{equation}
		y = \frac{1}{h(x)}\!\!\int\!\!\frac{h(x)}{g(x)}\!\!\int\!\! g(x)R(x)\, dx dx + \frac{C_1}{h(x)}\!\!\int\!\!\frac{h(x)}{g(x)}\,dx + \frac{C_2}{h(x)}
		\label{72}
		\end{equation}
		where
		\begin{gather}
		g(x) = e^{\left(\frac{B}{2} \pm \frac{\sqrt{B^2-4C}}{2}\right)x}, \label{73} \\
		h(x) = e^{\left(\frac{B}{2} \mp \frac{\sqrt{B^2-4C}}{2}\right)x}. \label{74}
		\end{gather}
		\label{Col3}
	\end{corollary}

	\begin{proof}
		Consider the general second order linear ordinary differential equation given by Eqn.~\eqref{1} such that $B(x)$ and $C(x)$ are constant i.e. $B(x)=B$ and $C(x)=C$,
		\begin{equation}
		y_{(2)} + By_{(1)} + Cy = R(x).
		\label{75}
		\end{equation}
		By Theorem~\ref{Thm}, this equation has solution given as
		\begin{equation}
		y = \frac{1}{h(x)}\!\!\int\!\!\frac{h(x)}{g(x)}\!\!\int\!\!g(x)R(x) \,dxdx + \frac{C_1}{h(x)}\!\!\int \!\! \frac{h(x)}{g(x)} \, dx + \frac{C_2}{h(x)},
		\label{76}
		\end{equation}
		where
		\begin{gather}
		g(x) = e^{\int P(x) \, dx}, \label{77} \\ 
		h(x) = e^{\int Q(x) \, dx}, \label{78} \\
		P(x) + Q(x) = B(x), \label{79} \\
		Q'(x) = {Q(x)}^2 -B(x)Q(x) + C(x) \, \label{80} \text{and}
		\end{gather}
		$g(x)$ and $h(x)$ are the first and second Integrating Factors respectively.\\ \\
		Since the coefficients $B(x)$ and $C(x)$ are constants $(B(x)=B, C(x)=C)$, we can consider $P(x)$ and $Q(x)$ to be constants too such that $P(x)=P$, $Q(x)=Q$ and 
		\begin{equation}
		P'(x) = Q'(x) = 0.
		\label{81}
		\end{equation}
		Then Eqn.~\eqref{79} and Eqn.~\eqref{80} respectively transform into
		\begin{gather}
		P+Q=B, \label{82}\\
		PQ=C. \label{83}
		\end{gather}
		On solving Eqn.~\eqref{82} and Eqn.~\eqref{83}, we obtain
		\begin{gather}
		P = \frac{B}{2} \pm \frac{\sqrt{B^2-4C}}{2}, \label{84} \\
		Q = \frac{B}{2} \mp \frac{\sqrt{B^2-4C}}{2}. \label{85}
		\end{gather}
		Then we can express the Integrating Factors $g(x)$ and $h(x)$ as
		\begin{gather}
		g(x) = e^{\int P(x) \,dx} = e^{\int P \,dx} = e^{\left(\frac{B}{2} \pm \frac{\sqrt{B^2-4C}}{2}\right)x}, \label{86} \\
		h(x) = e^{\int Q(x) \,dx} = e^{\int Q \,dx} = e^{\left(\frac{B}{2} \mp \frac{\sqrt{B^2-4C}}{2}\right)x}, \label{87}
		\end{gather}
		such that on using the Integrating factors given by the above equations in the solution expression given by Eqn.~\eqref{76}, the complete  solution of the required differential equation \eqref{75} is determined.
	\end{proof}

	The fourth result deals with the general second order linear ordinary differential equation and gives the formula for the complete solution when one of the complementary solutions is known. The result is given in form of a corollary of Theorem \ref{Thm} presented along with the proof here as follows.
	
	\begin{corollary}
		If one of the complementary solution of the general second order linear ordinary differential equation \eqref{1} is given as $f(x)$, then the solution of the differential equation \eqref{1} is given as
		\begin{equation}
		\begin{split}
		y &= f(x) \!\! \int \!\! \frac{e^{-\int \!\! B(x) \, dx}}{{f(x)}^2} \!\! \int \!\! e^{\int \!\! B(x) \, dx} f(x)R(x) \, dx dx\\ &+ C_1 f(x) \!\! \int \!\! \frac{e^{-\int \!\! B(x) \, dx }}{{f(x)}^2} \, dx + C_2 f(x).
		\label{88}
		\end{split}
 		\end{equation}
 		\label{Col4}
	\end{corollary}
	
	\begin{proof}
		Consider the general second order linear ordinary differential equation given by Eqn.~\eqref{1},
		\begin{equation}
		y_{(2)} + B(x)y_{(1)} + C(x)y = R(x).
		\label{89}
		\end{equation}
		By Theorem~\ref{Thm}, this equation has solution given as
		\begin{equation}
		y = \frac{1}{h(x)}\!\!\int\!\!\frac{h(x)}{g(x)}\!\!\int\!\!g(x)R(x) \,dxdx + \frac{C_1}{h(x)}\!\!\int \!\! \frac{h(x)}{g(x)} \, dx + \frac{C_2}{h(x)},
		\label{90}
		\end{equation}
		where
		\begin{gather}
		g(x) = e^{\int P(x) \, dx}, \label{91} \\ 
		h(x) = e^{\int Q(x) \, dx}, \label{92} \\
		P(x) + Q(x) = B(x), \label{93} \\
		Q'(x) = {Q(x)}^2 -B(x)Q(x) + C(x) \, \label{94} \text{and}
		\end{gather}
		$g(x)$ and $h(x)$ are the first and second Integrating Factors respectively.\\ \\
		If we have a function $f(x)$ such that it satisfies the homogeneous equation,
		\begin{equation}
		y_{(2)} + B(x)y_{(1)} + C(x)y = 0,
		\label{95}
		\end{equation}
		corresponding to the differential equation \eqref{89} then we can put $f(x)$ equal to one of the terms of complementary solution in Eqn.~\eqref{90}, i.e.
		\begin{equation}
		f(x) \equiv \frac{C_2}{h(x)}.
		\label{96}
		\end{equation}
		For the ease of solving we have taken the term of complementary solution corresponding to arbitary constant $C_2$ to be equivalent to $f(x)$ such that we can put $C_2=1$ to obtain the function $h(x)$ in terms of $f(x)$.

		On putting $C_2 = 1$, we get
		\begin{equation}
		f(x) = \frac{1}{h(x)} \implies h(x) = \frac{1}{f(x)}.
		\label{97}
		\end{equation}
		Now consider the following equations
		\begin{gather}
		g(x)h(x) = e^{\int (P(x) + Q(x)) \, dx} = e^{\int B(x) \, dx} \label{98} \\
		\implies g(x) = \frac{e^{\int B(x) \, dx}}{h(x)}. \label{99}
		\end{gather}
		Using Eqn.~\eqref{97}, we get
		\begin{equation}
		g(x) = f(x)e^{\int B(x) \, dx}.
		\label{100}
		\end{equation}
		Therefore, on using $g(x)$ and $h(x)$ as given in Eqn.~\eqref{97} and Eqn.~\eqref{100} in the solution expression in Eqn.~\eqref{90}, we get the solution of the second-order linear ordinary differential equation \eqref{89} as
		\begin{equation}
		\begin{split}
		y &= f(x) \!\! \int \!\! \frac{e^{-\int \!\! B(x) \, dx}}{{f(x)}^2} \!\! \int \!\! e^{\int \!\! B(x) \, dx} f(x)R(x) \, dx dx\\ &+ C_1 f(x) \!\! \int \!\! \frac{e^{-\int \!\! B(x) \, dx }}{{f(x)}^2} \, dx + C_2 f(x),
		\label{101}
		\end{split}
		\end{equation}
		when one of the complementary solutions of the differential equation \eqref{89} is given as $f(x)$.
	\end{proof}

	\bibliographystyle{plain}
	\bibliography{main}

\end{document}